\documentclass[12pt,a4paper]{amsart}
\usepackage{euscript,amsfonts,amssymb,amsmath,amscd,amsthm,enumerate,hyperref,mathrsfs}

\usepackage{tikz,bm,float}
\usetikzlibrary{calc}
\colorlet{lgray}{white!85!black}
\colorlet{lred}{white!75!red}

\usepackage{comment}

\usepackage{graphicx}

\usepackage{color}

\usepackage[margin=1.1in]{geometry}
\newtheorem{theorem}{Theorem} %[section]
\newtheorem*{theorem*}{Theorem}

\newtheorem{proposition}[theorem]{Proposition}

\newtheorem{corollary}[theorem]{Corollary}

\theoremstyle{remark}
\newtheorem{remark}[theorem]{Remark}

\numberwithin{equation}{section} \numberwithin{theorem}{section}

\newcommand{\N}{\mathbb N}

\newcommand{\Z}{\mathbb Z}

\newcommand{\R}{\mathbb R}

\usepackage{MnSymbol}

\usepackage{skak}

\usepackage{adjustbox}
\usetikzlibrary{arrows}
\usetikzlibrary{decorations.pathreplacing}

\title[Interacting particle systems and Hecke algebras]
{Interacting particle systems and random walks on Hecke algebras}

\author{Alexey Bufetov}

\address[Alexey Bufetov]{Hausdorff Center for Mathematics \& Institute for Applied Mathematics, University of Bonn, Germany. E-mail: alexey.bufetov@gmail.com}

\begin{document}

\begin{abstract}
	
In this paper we show that a variety of interacting particle systems with multiple species can be viewed as random walks on Hecke algebras. This class of systems includes the asymmetric simple exclusion process (ASEP), M-exclusion TASEP, ASEP(q,j), stochastic vertex models, and many others. As an application, we study the asymptotic behavior of second class particles in some of these systems.

\end{abstract}

\maketitle

\section{Introduction}

%\subsection*{Overview}
The connection of asymmetric simple exclusion processes (ASEP) and Hecke algebras goes back to \cite{ADHR} who noticed that the generators of ASEP satisfy Hecke algebra relations. Since then the connection was studied, extended and used in various contexts, see e.g. \cite{M1}, \cite{L1}, \cite{CdGW}. In this paper, we treat  this connection from a slightly different point of view by considering random walks on Hecke algebras (Section \ref{sec:2}). We believe that there are two significant advantages of this new point of view.

First, the concept of random walks on Hecke algebras for arbitrary Coxeter groups naturally generates not only ASEP, but a variety of other interacting particle systems (Section \ref{sec:3}). This might lead to a certain unification of methods for studying these systems with various interaction rules and boundary geometries. The structure of a random walk on a Hecke algebra might be viewed as the source of the integrability of the models. One immediate application is that all these models inherit an explicit reversible stationary measure (see Section \ref{sec:Mal}).
%Of course, this source is related to other sources, in particular, to the structure coming from the quantum groups via the Schur-Weyl duality; n

Second, from this point of view the Hecke algebra itself, or, equivalently, its faithful representation, plays a central role, while in literature one often starts by studying the relevant smaller representations of the algebra directly. The advantage of our approach is that one can use the structure of the Hecke algebra which is not visible for smaller representations. In fact, this project started from the attempt to understand the algebraic origin of a certain symmetry of interacting particle systems \cite[Theorem 1.3]{AHR}, \cite[Lemma 3.1]{AAV}, \cite[Theorem 1.6.1]{BW}, \cite[Proposition 2.1]{BB}. While in these papers the symmetry was established just by induction arguments, all these papers have non-trivial asymptotic applications of this property. In this paper we show that these symmetries are particular cases of the well-known involution in the Hecke algebra (Proposition \ref{prop:CPsymmetry}; this was also independently noticed by P. Galashin). This allows to conclude that \textit{all} interacting particle systems appearing as random walks on Hecke algebras have this symmetry. We use this symmetry for the asymptotic analysis of the second class particle in the half-line ASEP (Theorems \ref{th:TASEPalpha} and \ref{th:05lineTASEP}) and q-TAZRP (Theorem \ref{th:qTAZRPsecClass}). The applications of this symmetry first relate two processes on Hecke algebra, and then both processes are projected to (different) smaller representations. It seems very hard to relate these two smaller representations directly, without the use of the full Hecke algebra.

The main goal of this paper is to emphasize the role of random walks on Hecke algebras as a useful concept for the study of interacting particle systems. We hope that many more applications can be obtained with the use of this point of view. 

\subsection*{Acknowledgments} 
I am grateful to M.~Balazs, A.~Borodin, P.L.~Ferrari, J.~Kuan, A.~Povolotsky, and O.~Zaboronski for useful discussions.  
The work was partially supported by the Deutsche Forschungsgemeinschaft (DFG, German Research Foundation) under Germany's Excellence Strategy -- EXC 2047 ``Hausdorff Center for Mathematics''.

\section{Random walks on Hecke algebras}
\label{sec:2}

\subsection{Hecke algebras}

We briefly recall some facts about Hecke (or Iwahori-Hecke) algebras for Coxeter groups. See e.g. \cite{Hum} for basic definitions and proofs.

Let $(W,S)$ be a Coxeter group with a matrix $\{ m(s_i,s_j) \}_{s_i, s_j \in S}$. This means that $W$ is a group generated by a set $S \subset W$; the generators $s_i \in S$ satisfy relations $(s_i s_j)^{m(s_i, s_j)} = e$, where $e$ is the identity element, $m(s_i,s_j) \in \N$, $m(s,s)=1$, and $m(\tilde s,s)= m( s,\tilde s) \ge 2$. We will also assume that $S$ is finite.

Each $w \in W$ can be written in the form $w=s_r s_{r-1} \dots s_1$ for some sequence of $s_i$'s from $S$. The length of $w$ (notation $l(w)$) is the smallest possible $r$ in such a decomposition. The length of the identity element is set to be 0. Length satisfies natural properties, in particular, for any $s \in S$, $w \in W$, we have either $l(s w) = l(w)+1$ or $l(sw) = l(w)-1$.

A \textit{Hecke algebra} $\mathcal H (W)$ is the algebra with a linear basis $\{ T_w \}_{w \in W}$ and the multiplication which satisfies the following rules for any $s \in S$, $w \in W$,
\begin{equation}
\label{eq:HeckeRules}
\begin{cases}
T_s T_w = T_{sw}, \qquad & \mbox{if $l(sw)=l(w)+1$}  \\
T_s T_w = (1-q) T_w + q T_{sw}, \qquad & \mbox{if $l(sw)=l(w)-1$}.
\end{cases}
\end{equation}
It is clear that such rules can be used for a computation of the product $T_{w_1} T_{w_2}$ for any $w_1, w_2 \in W$; a non-trivial (but very well-known) part is that the rules are consistent and indeed define an (associative) multiplication.

We assume that $q \in \R$. In most of our constructions we will additionally assume that $0 \le q \le 1$. Note that our convention in \eqref{eq:HeckeRules} does not correspond to the most standard convention (usually one writes the factor $(q-1)$ instead of $(1-q)$ in the second equation); our choice is motivated by probabilistic applications.

Let $\mathfrak i: \mathcal H (W) \to \mathcal H (W)$ be a linear map such that $\mathfrak i \left( T_w \right) = T_{w^{-1}}$. The following proposition is well-known (and can be straightforwardly proved by induction in $l(w)$ with the use of \eqref{eq:HeckeRules}).

\begin{proposition}
\label{prop:CPsymmetry}
The map $\mathfrak i$ is an involutive anti-homomorphism. In more detail, for any $T_1, \dots T_r \in \mathcal H (W)$ we have
\begin{equation*}
\mathfrak i \left( T_r T_{r-1} \dots T_2 T_1 \right) = \mathfrak i \left( T_1 \right) \mathfrak i \left( T_2 \right) \dots \mathfrak i \left( T_{r-1} \right) \mathfrak i \left( T_r \right),
\end{equation*}
and also, trivially, $\mathfrak i^2 \left( T_1 \right) = T_1$.

\end{proposition}

\begin{corollary}
If $T_1, \dots, T_r \in \mathcal H (W)$ are such that $\mathfrak i \left( T_j \right) = T_j$, for any $1 \le j \le r$, then
\begin{equation*}
\mathfrak i \left( T_r T_{r-1} \dots T_2 T_1 \right) = T_1 T_2  \dots T_{r-1} T_r.
\end{equation*}
\end{corollary}

\subsection{Random walks on Hecke algebras}
\label{sec:randomWalksGeneral}

Let us set
$$
\mathcal{H}_{prob} (W) := \{ h \in \mathcal{H} (W) : h = \sum_{w \in W} \kappa_w T_w, \sum_{w \in W} \kappa_w = 1, \kappa_w \ge 0 \}.
$$
It is clear that $\mathcal{H}_{prob} (W)$ is closed under the multiplication and taking convex combinations of its elements. 
%The elements of $h \in \mathcal{H}_{prob} (W)$ are in natural bijection with probability measures on $W$.

Consider first a general setting. Let $H_1, H_2, \dots, $ be an arbitrary random sequence of elements from $\mathcal{H}_{prob} (W)$. We will call the sequence $W_r := H_r H_{r-1} \dots H_1$ a \textit{random walk on the Hecke algebra} $\mathcal H (W)$. Next, consider the decomposition into the linear basis $W_r = \sum_{w \in W} \kappa_r (w) T_w$, and set
$$
\mathcal P_r (w) := \kappa_r (w)
$$
The collection of $\{ \mathcal P_r (w) \}_{w \in W}$ is a probability measure on $W$, and this process is a Markov (in discrete time $r$) stochastic process.
%Further on, if the (random) coefficients defined via
%$$H_r =: \sum_{w \in W} h_r (w) T_w, $$
%satisfy $ \sum_{w \in W} h_r (w) \equiv 1$ and $h_r (w) \ge 0$, then we have
%$\sum_{w \in W} \kappa_r (w) = 1$, and the process $\{ W_r \}$ is Markovian.
We refer to these processes as \textit{random walks on Coxeter groups generated by Hecke algebras}. With obvious modifications of notation, one can similarly define continuous time processes $W_t$, $\mathcal P_t (w)$, $t \in \R_{>0}$.

A natural class of examples appears if $\{ H_i \}$ are assumed to be independent identically distributed random elements.
As a specific example, let $H_i$ be a sequence of i.i.d. uniformly random elements from the set $\{ T_s \}_{s \in S}$ (recall that we assume that $S$ is finite). A continuous time version of this process would correspond to assigning to any $T_s$ an independent Poisson process in $\R_{>0}$, and applying $T_s$ whenever a point arrives from this Poisson process. We call these processes the \textit{discrete time ASEP on} $W$ and the \textit{continuous time ASEP on} $W$, respectively. This is a natural name since in the case $W=S_n$ (symmetric group), we will recover the standard multi-species ASEP (see also Section \ref{sec:3-1}).

%something else ?

\subsection{Mallows measure}
\label{sec:Mal}

%Choose the constants $a_s$, $b_s$ from \eqref{eq:HeckeRules} as $a_s \equiv 1-q$, $b_s \equiv q$, with $q \in \R_{\ge 0}$. 

Assume that the series $\sum_{w \in W} q^{-l(w)}$ converges, and denote its sum by $Z=Z(W,q)$. Note that this series converges if $W$ is finite or if $q>1$.
Define the \textit{Mallows measure} of $W$ via
$$
\mathcal M (w) := \frac{q^{-l(w)}}{Z}, \qquad w \in W.
$$

\begin{proposition}
The Mallows measure is a reversible stationary measure of any random walk on $W$ generated by its Hecke algebra.
\end{proposition}
\begin{proof}
The detailed balance equation is immediate from the definitions.
\end{proof}

For $q>1$ the measure is concentrated near the identity, while if $q<1$ and $W$ is finite, then the measure concentrated near the longest word. There is a meaningful limit $|W| \to \infty$, at least in cases when $W$ is the Weyl group of type A (\cite{GO1}, \cite{GO2}) and B (\cite{K}).

We will need a more detailed description of a Mallows measure on $S_n$ and assume that $q \in [0;1)$. It will be convenient to think about it as a random permutation of a collection of elements with the prescribed linear order on them $a_1 < a_2 < \dots < a_n$. The probability of a permutation $a_{c_1} a_{c_2} \dots a_{c_n}$ is proportional to $q^{n(n-1)/2 - \mathcal{N}\left( \{c_j\} \right)}$, where $\mathcal{N}\left( \{c_j\} \right)$ denotes the number of inversions in the permutation $i \mapsto c_i$, $1 \le i \le n$. Thus, the most probable word is the word $a_n a_{n-1} \dots a_1$. Whenever a certain collection of linearly ordered elements is distributed according to this Mallows measure, we will say that this collection is \textit{in q-equilibrium}.

There is an important sampling algorithm which allows to produce the random permutation of $a_i$'s which is in q-equilibrium. This algorithm is essentially due to Mallows \cite{M}, see e.g. \cite[Section 3]{GO1} for a detailed exposition. For $m=1,2,\dots$ denote by $G_{q,m}$ the independent truncated geometric random variables with the distribution
$$
\mathrm{Prob} \left( G_{q,m} = z \right) = \frac{q^{z-1} (1-q)}{1-q^m}, \qquad z=1, \dots m.
$$
The algorithm works as follows: write the word $a_n a_{n-1} \dots a_1$ (in this specific order), and take the letter which is in the position $G_{q,n}$ from the left (so this is $a_{n-G_{q,n}+1}$ on the first step). Denote by $\xi_1$ this letter and delete it from the initial ordered word. Next, take the letter which is in a position $G_{q,n-1}$ from the left in the ordered word. When we count positions from left to right, we do not count positions of deleted letters. Denote by $\xi_2$ the obtained letter, delete it from the word, and continue the procedure. On $k$-th step we use random variable $G_{q,n-k+1}$ and choose one of remaining $(n-k+1)$ letters according to the same rule. This will give us the random permutation $\xi_1 \xi_2 \dots \xi_n$, and the claim is that it is distributed according to the Mallows measure.

An important corollary of this algorithm is that $\xi_1$ depends on $G_{q,n}$ only, which allows to compute its distribution. Analogously, $\xi_1$, .., $\xi_s$ depend only on the first $s$ steps of this algorithm, which simplifies the computations of their joint distribution. Finally, note that due to the symmetry of the problem one can run the same algorithm for positions not from left to right, but from right to left, which allows, for example, to compute the distribution of $\xi_n$.

In \cite{GO1} the extension of Mallows measures to infinite permutations $\Z_{\ge 0} \to \Z_{\ge 0}$ was obtained. They can be sampled by the similar algorithm (see \cite[Section 4]{GO1}). In fact, the algorithm is even simpler, since the truncated geometric random variables in the construction above are replaced by the standard geometric random variables with parameter $q$. This can be visible as the limit $n \to \infty$ of the construction above, since $G_{q,n}$ becomes the standard geometric random variable as $n \to \infty$.

\section{Interacting particle systems generated by random walks on Hecke algebras}
\label{sec:3}

In Sections \ref{sec:3-1}--\ref{sec:gen-M-exc} $W$ will be the symmetric group $S_N$. In Section \ref{sec:12case} $W$ will be a hyperoctahedral group. Another important case appears when $W$ is an affine Weyl group. It corresponds to particle processes on a ring, and will be considered elsewhere.

For each process we describe the set of generators of a random walk on a Hecke algebra. This allows to define continuous / discrete time variations of these processes.

\subsection{Multi-species ASEP}
\label{sec:3-1}

Consider a multi-species ASEP on $\mathcal H (S_n)$ defined in Section \ref{sec:randomWalksGeneral} with $q \in [0;1]$. It is a tautology that the evolution from our definition coincides with the description via update rules in the multi-species ASEP, see e.g. \cite[Section 2.1]{AAV}, \cite[Section 3]{BB}. In probabilistic terms, the resulting dynamics lives on the space consisting of $N$ positions, and contains particles on $N$ types, so for each
type there is exactly one particle of this type. The most standard version of ASEP involves only two types, which are referred to as particles and holes, and one often considers the whole $\Z$ as the space of positions rather than its finite subset. The number of positions can be made infinite and the number of types can be reduced to two with the help of the next two remarks.

\subsection{Remark on an infinite space}

While all constructions of this section describe the interacting particle systems on a finite space, they can be extended to the infinite case by the standard argument of Harris (\cite{H1}, \cite{H2}). Namely, all the processes under discussion on the infinite space at any fixed time with probability 1 are (infinite) collections of processes on a finite space which did not interact with each other, see also \cite[Proof of Theorem 1.4]{AAV}, \cite[Proof of Theorem 3.1]{BB}. Thus, the results about infinite space processes will follow from the results about finite space versions.

\subsection{Remark on the number of types}

One can consider the projection of random walks on Hecke algebras to cosets of parabolic subgroups in $W$. Probabilistically, this means an identification of some of types of particles. In particular, one can recover in such a way the most studied interacting particle systems which contain only two types: particles and holes, and the systems with first class particles, second class particles, and holes.

\subsection{Stochastic six vertex model}

Define $Y_{s, x} := x T_s + (1-x) T_{e}$, where $s \in S$, $x \in [0;1]$, and $e$ is an identity element in $W=S_N$. A multiplication by $Y_{s,x}$ corresponds to a stochastic vertex in the multi-color stochastic six vertex model, see \cite[Figure 1.5.8]{BW}. If one is interested in a rectangular lattice, then the configuration in the rectangle of height 1 will be produced by the element
$$
W_{a,b} := Y_{(b-1,b),x} \dots Y_{(a+1,a+2),x} Y_{(a,a+1),x},
$$
and in order to obtain an arbitrary rectangle, one needs to consider the multiplication of such elements:
$$
W_{a-k,b-k} \dots W_{a-1,b-1} W_{a,b},
$$
which defines a configuration of the stochastic six vertex model in the rectangle $(b-a+1)$ by $(k+1)$. Note that other arrangement of multiplications will lead to other arrangement of lines of the stochastic six vertex model. Also note that the random walk on Hecke algebra in this case is deterministic. Nevertheless, the random walk on $W$ generated by Hecke algebra is non-trivial and actually more complicated than the continuous time ASEP dynamics.

\subsection{ASEP(q,M)}
\label{sec:asep-qj}

In this section we consider a random walk on the Hecke algebra $\mathcal H( S_n)$ which produces the so called multi-species ASEP(q,M) process. Its single species version was introduced in \cite{CGRS} and the multi-species analog was introduced in \cite{K2}. For this, we will need to introduce certain elements in $\mathcal H( S_n)$ related to Mallows measures. As before, we fix $q \in [0;1)$.

For integers $1 \le a<b \le n$ denote by $[a;b] := \{ j \in \Z : a \le j \le b \}$ the interval between $a$ and $b$, and by $S_{a;b} \subset S_n$ the subgroup which permutes the elements from $[a;b]$ only. Define
$$
\mathcal{M}_{a;b} := \sum_{w \in S_{a;b}} q^{(b-a+1)(b-a)/2 - \mathcal{N}\left( w \right)} T_w, \qquad \mathcal{M}_{a;b} \in \mathcal H( S_n),
$$
where $\mathcal{N}\left( w \right)$ is the number of inversions in $w$. The main property of the element $\mathcal{M}_{a;b}$ is
$$
T_{w} \mathcal{M}_{a;b} = \mathcal{M}_{a;b} T_{w} = \mathcal{M}_{a;b}, \qquad \mbox{for any $w \in S_{a;b}$}.
$$
The multiplication by the element $\mathcal{M}_{a;b}$ can be thought of as bringing the existing (random) configuration to q-equilibrium inside the interval $[a;b]$ without changing anything outside of this interval.

Let $n = N M$, with $M,N \in \Z_{>0}$, and consider the following set of generators of a random walk on the Hecke algebra :
$$
\left\{ \mathcal{M}_{(x-1)M+1;xM} \mathcal{M}_{xM+1;(x+1)M} T_{(xM, xM+1)} \mathcal{M}_{(x-1)M+1;xM} \mathcal{M}_{xM+1;(x+1)M} \right\}_{x=1}^{N-1}.
$$
It is convenient to think about configurations arising in this process as consisting of $N$ blocks numbered from left to right, with each block having $M$ particles. Applying the generator element corresponding to $x \in [1;N-1]$, we bring into the q-equilibrium the blocks number $x$ and $x+1$, then these blocks interact in one place (the right-most particle from block $x$ interact via $T_{(xM, xM+1)}$ with the left-most particle from the block $x+1$), and these blocks are brought into q-equilibrium again. Thus, after applying of at least one of generators a block inside the configuration is in q-equilibrium, and will remain in q-equilibrium during the rest of the dynamics. However, the elements inside it might change due to the interaction between blocks.

This dynamics is equivalent to the multi-species ASEP(q,M) from \cite{CGRS}, \cite{K2} (there the notation is ASEP(q,j), with $2j=M$). Indeed, let us consider a single species process. This means that types of particles from 1 to N are split in a monotone way into two types, one for particle and another for holes. Since all blocks participating in the dynamics are in q-equilibrium, we can think of any block as one position which may contain from 0 to $M$ particles. The distribution of these particles inside the block is always independent of the rest of the dynamics.
Assume that the block number $x$ has $n_1$ particles (and $M-n_1$ holes), and the block number $x+1$ has $n_2$ particles (and $M-n_2$ holes) after such an identification. The action of a generator on these two blocks can lead to the three situations: One particle jumps from left to right, one particle jumps from right to left, or nothing changes. The first scenario happens if the rightmost position in block $x$ is occupied by one of $n_1$ particles, and the leftmost position in block $x+1$ is occupied by one of $M-n_2$ holes. The sampling algorithm of Mallows measure (see Section \ref{sec:Mal}) implies that these probabilities are equal to $(1-q^{n_1})/(1-q^M)$ and $(1-q^{M-n_2})/(1-q^M)$, respectively. Thus, the action of a generator leads to a jump of particle from block $x$ to block $x+1$ with probability
$$
\frac{(1-q^{n_1}) (1-q^{M-n_2})}{(1-q^M)^2}.
$$

The probability of a jump from the block $x+1$ to $x$ is the product of the probability that the leftmost position in block $x+1$ is occupied, the rightmost position in block $x$ is vacant, and the factor $q$ due to the asymmetry of the model. This produces the probability
$$
q^{M-n_2+n_1+1} \frac{(1-q^{n_2})(1-q^{M-n_1})}{(1-q^M)^2}.
$$
These probabilities are up to a constant (not depending on $n_1$ and $n_2$) factor the same as given in \cite[Section 3.1]{CGRS}. Similar and a bit more involved computations based on the sampling of Mallows measures allow to show that the rules for the second, third etc. class particles coincide with those from \cite{K2} (another way is to compare with the combinatorial description presented in \cite{K3}). 

Note that ASEP(q,M) has several important degenerations. First, obviously, for $M=1$ we recover the usual ASEP. For $M \to \infty$ we obtain the q-totally asymmetric zero range process (q-TAZRP). We will use this degeneration in Section \ref{sec:2qTazrp} below. Also, $q=0$ case gives rise to the $M$-exclusion TASEP. The construction from this section supplies the Hecke algebra structure for multi-species versions of these processes as well.

\subsection{General M exclusion asymmetric process}
\label{sec:gen-M-exc}

The key idea of the construction from the previous section is to use the elements $\mathcal{M}_{xM+1;(x+1)M}$ which bring the system to a q-equilibrium inside the blocks. This allows to treat positions inside a block as one position which contains $M$ particles of different types; their respective positions inside the block are governed by the Mallows measure and are independent of the rest of the configuration. One can use the same idea to construct many more processes, in which more than one particle is allowed to jump from one block to another.

Let us fix an arbitrary element $Y_{1;2M} \in \mathcal{H} \left( S_{1;2M} \right)$. For any $sh \in \Z$ the subalgebra $\mathcal{H} \left( S_{1;2M} \right)$ is naturally isomorphic to the subalgebra $\mathcal{H} \left( S_{sh+1;sh+2M} \right)$ via the shift of indices; denote by $Y_{sh+1;sh+2M}$ the image of $Y_{1;2M}$ under this isomorphism.

Let $n = N M$, with $M,N \in \Z_{>0}$, and consider the following set of generators of a random walk on the Hecke algebra :
$$
\left\{ \mathcal{M}_{(x-1)M+1;xM} \mathcal{M}_{xM+1;(x+1)M} Y_{((x-1)M+1, (x+1)M)} \mathcal{M}_{(x-1)M+1;xM} \mathcal{M}_{xM+1;(x+1)M} \right\}_{x=1}^{N-1}.
$$
Similarly to the previous section, the $\mathcal{M}$ factors guarantee that the blocks of size $M$ can be interpreted as one position with at most $M$ particles, while the element $Y_{((x-1)M+1, (x+1)M)}$ produces the rules how the neighboring blocks can exchange particles.
In particular, if we consider the projection to a single-species process, several particles might be allowed to jump forward or backwards (the specific probabilities are computed with the use of the element $Y_{((x-1)M+1, (x+1)M)}$ and algorithms for sampling Mallows measures). 
This scheme produces a variety of known and new models of multi-species interacting particle systems. Due to the construction, all these models have the type-position symmetry by Proposition \ref{prop:CPsymmetry}. 

\subsection{Half-line case}
\label{sec:12case}

Recall that the hyperoctahedral group can be viewed as a subgroup of permutations of the set $\{-N,\dots, ,-2,-1,1,2, \dots, N\}$ with the special property $\pi(-i) = -\pi(i)$, for all $i =1,2,\dots, N$. The general formalism of Section \ref{sec:2} defines the multi-species asymmetric exclusion process on this group. As in the symmetric group case, this process has natural interpretation in terms of interacting particles. Namely, we consider the space of positions $\{1,2,\dots, N\}$ which can be occupied by particles of types $\{-N,\dots, -1, 1, \dots, N\}$; however, types $i$ and $(-i)$ cannot be present in the system simultaneously. The multiplication by basis elements $T_s$, $s \in S$, results in the following interaction rules: In the ``bulk'', for positions $(n,n+1)$, $n \ge 1$, the rules are the same as for the multi-species ASEP from Section \ref{sec:3-1}. At boundary position $1$, the type of particle can change sign (with probabilities coming from the multiplication in the Hecke algebra). 

Thus, ASEP for the hyperoctahedral group produces multi-species ASEP with a boundary, and also produces natural rules for injecting / deleting various types of particles at the boundary. By projecting this process to parabolic subgroups, we obtain the rules for single-species / several species processes with the boundary. 

\begin{remark}
For simplicity, we used only one asymmetry parameter $q$ in \eqref{eq:HeckeRules}. However, Hecke algebras might have more asymmetry parameters; in particular, for BC root systems they have two asymmetry parameters $q_1$ and $q_2$, one of them corresponds to the asymmetry in the bulk, and another one --- at the boundary. The random walks on such Hecke algebras correspond to interacting particle systems in the same way. It seems that a particularly natural case from the probabilistic point of view will be to set to 0 the boundary asymmetry parameter, while keeping the bulk asymmetry parameter general. This leads to simpler rules for injection of particles at the boundary while still preserving the general ASEP (rather than TASEP) dynamics at the bulk.  
\end{remark}

\section{Asymptotic applications}
\label{sec:4}

\subsection{Second class particle in the semi-infinite ASEP}

In this section we consider the semi-infinite ASEP on $\Z_{\ge 1}$ with a source which inject or delete particles from 1.
Let us start with a process which involves only one class of particles (and holes). This is a continuous time process with empty initial configuration. In the bulk particles jump to the right with rate $1$ and jump to the left with rate $q$. At the boundary position $1$ they are injected with rate $\alpha \in \R_{>0}$ and deleted with rate $q \alpha$. All of these jumps are subject to the exclusion rule. Let $\eta_t (x)$ be a configuration at time $t \in \R_{\ge 0}$. Liggett \cite{L} proved the following description of these dynamics.

\begin{theorem}
\label{th:Liggett}
As $t \to \infty$, the dynamics converges to the stationary measure on $\Z_{\ge 1}$, where the convergence is in the sense of finite dimensional distributions. For $\alpha \le 1/2$ the limiting stationary measure is the product Bernoulli measure with parameter $\alpha$. For $\alpha > 1/2$, the limiting stationary measure is denoted by $\mu_{1/2}^{\alpha}$ and has a more complicated structure; however, asymptotically (``in the bulk'') it becomes the product Bernoulli measure with parameter 1/2.
\end{theorem}

With the use of the matrix product ansatz, Grosskinsky \cite{G} described the correlation functions of $\mu_{1/2}^{\alpha}$ in the totally asymmetric case $q=0$.
\begin{theorem}
\label{th:Grosskinsky}
Assume that $q=0$. For $z \in \Z_{\ge 2}$ we have the following expressions for the probability to find a particle in a given position:
$$
\mu_{1/2}^{\alpha} \left( \eta \in \{0,1 \}^{\Z_{\ge 1}}: \eta (z) = 1 \right) = \frac{1}{4^z} \sum_{k=2}^z \frac{(2(z-1) -k)! (k-1)}{(z-1)! (z-k)!} \left( (1+k) 2^k - \alpha^{-k} \right) =: \rho_{\alpha} (z),
$$
and
$$
\mu_{1/2}^{\alpha} \left( \eta \in \{0,1 \}^{\Z_{\ge 1}}: \eta (1) = 1 \right) = 1 - \frac{1}{4 \alpha} =: \rho_{\alpha} (1).
$$
Moreover, one has
$$
\mu_{1/2}^{\alpha} \left( \eta \in \{0,1 \}^{\Z_{\ge 1}}: \eta (z) = 1, \eta (z+1)=1, \dots, \eta(z+m-1) = 1 \right) = \frac{1 + (2 \rho_{\alpha} (z)-1) m}{2^m}.
$$
\end{theorem}
As noted in \cite{G}, it is possible to see that $\rho(z) \to 1/2$ as $z \to \infty$, which allows to recover the Liggett's description.

Let $q=0$, and consider now TASEP on a half-line with two types of particles (and one type of holes). As an initial configuration, set $\eta (1)=\eta (2)= \dots = \eta (k-1) = 1$, $\eta(k)=\dots=\eta(l-1)= +\infty$, $\eta(l)=2$ (1 stands for the first class particles, 2 stands for the second class particles, and $+\infty$ stands for holes), and the rest of positions is filled by holes as well. Next, the interaction of particles happen as in the usual TASEP, and we have special rules for the boundary: namely, we postulate that if position 1 contains the second class particle or a hole, then the first class particle can be injected into 1, and the second class particle (or a hole) disappears from the system. Note that only first class particles can be injected, so while the second class particle is initially in the system, it can disappear. We will be interested in the probability that this happens.

\begin{theorem}
\label{th:TASEPalpha}
For TASEP on an infinite half-line with a source, in the notations and assumptions above, the probability that the second class particle exits the system is equal to
$$
\frac{1 + (2 \rho_{\alpha} (k)-1) (l-k)}{2^{l-k+1}}, \qquad \mbox{if $\alpha \ge 1/2$},
$$
and it is equal to $\alpha^{l-k+1}$, if $\alpha \le 1/2$.
\end{theorem}

\begin{proof}

Consider a homogeneous continuous time multi-species ASEP on the hyperoctahedral group $B_N$ as defined in Section \ref{sec:12case}. As an initial element let us take the permutation $(l-1,l) (l-2,l-1) \dots (k,k+1)$, which satisfies $\pi(k)=l$, $\pi (i)=i-1$, for $k+1 \le i \le l$, and $\pi(i)=i$ for all other $i \in \Z_{\ge 1}$. Let $\pi_t^{hTASEP}$ be the random element after time $t$ reached by the random walk. By regarding all elements $\le k-1$ as the first class particles, $k$ as the second class particle, and elements $>k$ as holes, we obtain that the quantity in question is equal to
$$
\lim_{t \to \infty} \mathrm{Prob} \left( \pi_t (k) < 0 \right).
$$
By applying Proposition \ref{prop:CPsymmetry} we obtain that this limit is equal to
\begin{equation}
\label{eq:05lineInv}
\lim_{t \to \infty} \mathrm{Prob} \left( \tilde \pi_t^{-1} (k) < 0 \right),
\end{equation}
where the configuration $\tilde \pi_t$ is constructed as a result of the two step procedure: First, we run the (homogeneous continuous time) multi-species ASEP on $B_N$ for time $t$, and then do updates in $(l-1,l)$, $(l-2, l-1)$, ..., $(k,k+1)$ (in this order). If we will consider elements $<0$ as particles, and the elements $>0$ as holes, we obtain that $\pi_t^{-1} (k) < 0$ if and only if all positions from $k$ to $l$ are filled by particles after the continuous time process and before the discrete updates. Applying Theorems \ref{th:Liggett} and \ref{th:Grosskinsky}, we obtain the statement.
\end{proof}

Let us now discuss the model for a general $q \in [0;1)$. Since the source has nontrivial asymmetry, one needs to introduce the opportunity for the second class particle to jump back into the system. It is natural to do this in the following way: assume that we have two special particles --- the second class particle and the third class particle. Exactly one of them is in the system at any time. Inside the bulk the particles interact as in ASEP. At the boundary, if the second class particle is in position $1$, then it is replaced by the third class particle with rate $\alpha q$, while if the third class particle is at position $1$, then it is replaced by the second class particle with rate $\alpha$. Similarly, the first class particle at $1$ is replaced by a hole with rate $\alpha q$ and the hole at $1$ is replaced by the first class particle with rate $\alpha$.

With such definitions, the second class particle can exit the system and then enter again. We will be interested in probability that the second class is in the system after large time $t$. We will assume that $\alpha \le 1/2$, because the detailed information (comparable with Theorem \ref{th:Grosskinsky}) about $\mu_{1/2}^{\alpha}$ is not known for $q \ne 0$ case.

\begin{theorem}
\label{th:05lineTASEP}
Consider open ASEP with the assumptions above and the initial configuration $\eta_0 (1)=\eta (2)= \dots = \eta (k-1) = 1$, $\eta(k)=\dots=\eta(l-1)= +\infty$, $\eta(l)=2$, and the rest of positions is filled by holes as well. Let $p_t$ be the probability that the
second class particle is in the system after time $t$. Then we have
$$
\lim_{t \to \infty} p_t = \alpha \left( \alpha + (1-\alpha)q \right)^{l-k}.
$$
\end{theorem}
\begin{remark}
Recall that we consider $\alpha \le 1/2$ case only. When $q=0$ we recover the statement from Theorem \ref{th:05lineTASEP}.
\end{remark}
\begin{proof}
The proof is analogous to the proof of Theorem \ref{th:05lineTASEP} up to the point of obtaining equation \eqref{eq:05lineInv}, and a new part is to compute the probability that the position $k$ is filled by a particle in the reversed time process (so the discrete updates are at the end). Theorem \ref{th:Liggett} implies that before the discrete updates we obtain the Bernoulli product measure in the limit $t \to \infty$ with parameter $\alpha$. Then we apply to it the updates: first at positions $(l-1,l)$, next $(l-2,l-1)$, ...., finally at $(k,k+1)$. One shows by an induction argument in $m=1,2,\dots$ that the probability to have a particle at position $(l-m)$ after updates from $(l-1,l)$ to $(l-m,l-m+1)$ is $\alpha \left( \alpha + (1-\alpha) q \right)^{m}$. This implies the result.
\end{proof}

\subsection{Second class particle in qTAZRP}
\label{sec:2qTazrp}

In this section we will consider q-totally asymmetric zero range process (qTAZRP). It was introduced by Sasamoto-Wadati \cite{SW}. An equivalent process q-TASEP was obtained by Borodin-Corwin \cite{BC} in the framework of dynamics on Macdonald processes. We will consider the following space of configurations for this process: Let $\Z_{\ge -1}$ be the set of positions. In each non-negative position one has finite amount of particles, and at $(-1)$ one has infinitely many particles. We will allow at most one second class particle in the configuration, and it must be in a non-negative position; in particular, this means that all particles at $(-1)$ are first class particles. If a certain position contains $l$ first class particles, then with rate $(1-q^l)$ one of these particles jump one position to the right. The rule for the unique second class particle is a bit more complicated --- if it is in a position with $l$ first class particles, then it jumps one position to the right with rate $q^l (1-q)$. The jumps from $(-1)$ to $0$ happen with rate 1, which corresponds to the case $l = \infty$ in these rules.

Let us start with describing known results about the evolution of the system with first class particles only. We would need to introduce some notation. For $\alpha \in [0;1)$ set
$$
\kappa (\alpha ) := \sum_{k=0}^{\infty} \frac{q^k}{(1- \alpha q^k)^2}.
$$
Note that $\kappa (\alpha )$ is a strictly increasing function with $\kappa (0 ) = 1/(1-q)$ and $\kappa (1) = +\infty$, so the inverse function $\alpha = \alpha (\kappa)$ is well-defined. The following theorem is a hydrodynamical limit for the step initial conditions which follows from general results of \cite{AV}, with a description of the limit density through the entropic solution to Burgers equation. The explicit limiting density in our case was computed in Ferrari-Veto \cite{FV}.

\begin{theorem}
\label{th:hydroQtazrp}
Consider qTAZRP with an initial configuration in which all non-negative positions are empty. Let $X_N (t)$ be the number of particles in position $N$ after time $t$. Then we have
$$
\lim_{N \to \infty} \mathrm{Prob} \left( X_{N} ( \kappa N ) = l \right) = (\alpha;q)_{\infty} \frac{\alpha^l}{(q;q)_l}, \qquad l \in \Z_{\ge 0},
$$
and the random variables $\{ X_{N+s} ( \kappa N ) \}_{s=0}^{m}$ are asymptotically independent and identically distributed for any fixed $m \in \Z_{\ge 1}$.
\end{theorem}

Let us now study the qTAZRP with one second class particle, and describe the behavior of this particle.

\begin{theorem}
\label{th:qTAZRPsecClass}
Consider qTAZRP with an initial configuration in which all non-negative positions are empty except of position $s$ which contains one second class particle and nothing else. Let $\mathfrak{h} (t)$ be the position of the second class particle after time $t$. Then
$$
\lim_{t \to \infty} \mathrm{Prob} \left( \frac{\mathfrak{h} (t)}{t} \ge \frac{1}{\kappa (\alpha)}  \right) =
\begin{cases}
\alpha, \qquad & s=0, \\
1 - \alpha^{2} ( 1- \alpha)^{s-1} \qquad & s \ge 1.
\end{cases}
$$
\end{theorem}
\begin{remark}
Since $\alpha = \alpha (\kappa)$ is a well-defined function, this theorem gives the limiting distribution for the scaled position of the second class particle. In particular, its support is $[0;(1-q)]$.
\end{remark}

\begin{proof}
We consider the random walk on Hecke algebras which generate qTAZRP/ASEP (q,M), see Section \ref{sec:asep-qj}. Applying Proposition \ref{prop:CPsymmetry}, we obtain
$$
\mathrm{Prob} \left( \mathfrak{h} (t) \ge xM \right) = \mathrm{Prob} \left( \hat \pi_t (xM) \le 0 \right),
$$
where $\hat \pi_t$ is the configuration obtained as a result of the continuous time qTAZRP up to time $t$ and doing discrete updates first at $((x-s)M, (x-s)M+1)$, next in $((x-s)M+1, (x-s)M+2)$, ..., finally at $(xM-1, xM)$. Sending $M$ to infinity, applying Theorem \ref{th:hydroQtazrp} and performing an analysis of arising cases related to discrete updates, we arrive at the statement of the theorem.
\end{proof}


\begin{thebibliography}{100}

\bibitem[ADHR]{ADHR} FC Alcaraz, M Droz, M Henkel, V Rittenberg, ``Reaction-diffusion processes, critical dy-
namics, and quantum chains'', Annals of Physics 230 (2), 250-302, (1994).

\bibitem[AAV]{AAV} G.~Amir, O.~Angel, B.~Valko, \textit{The TASEP speed process}, Ann. Probab.
 39 (2011), 1205--1242, {\tt arXiv:0811.3706}.

\bibitem[AHR]{AHR} O.~Angel, A.~Holroyd, D.~Romik, \textit{The oriented swap process}, Ann. Probab.
37 (2009), 1970--1998, {\tt arXiv:0806.2222}.

\bibitem[AV]{AV} E.~Andjel, M.~Vares, \textit{Hydrodynamic Equations for Attractive Particle Systems on Z}. J.
Stat. Phys.,47, no. 1/2, 265–288, (1987).

\bibitem[BB]{BB} A.~Borodin, A.~Bufetov, \textit{Color-position symmetry in interacting particle systems}, preprint, {\tt arXiv:1905.04692}.

\bibitem[BC]{BC} A.~Borodin, I.~Corwin, \textit{Macdonald Processes}, Probab. Theory and Related Fields 158 (2014),
 225--400, \texttt{arXiv:1111.4408}.

\bibitem[BW]{BW} A.~Borodin, M.~Wheeler, \textit{Coloured stochastic vertex models
and their spectral theory}, preprint, {\tt arXiv:1808.01866}.

\bibitem[CdGW]{CdGW} L.~Cantini, J.~de Gier, M.~Wheeler,
``Koornwinder polynomials and the stationary multi-species asymmetric exclusion process with open boundaries'', J. Phys. A: Math. Theor. 49 (2016), 444002 {\tt arXiv:1607.00039}

\bibitem[CGRS]{CGRS} G. Carinci, C. Giardina, F. Redig, T. Sasamoto, ``A generalized asymmetric exclusion process with $U_q (sl_2)$  stochastic duality'', Probab. Theory Relat. Fields, 166: 887--933, (2016).

\bibitem[CP]{CP} E. Cator, L. Pimentel, \textit{Busemann functions and the speed of a second class particle
in the rarefaction fan}, Ann. Probab. 41 (2013), 2401--2425, {\tt arXiv:1008.1812}.

\bibitem[FK]{FK} P. A. Ferrari, C. Kipnis, \textit{Second class particles in the rarefation fan}, Ann. Inst. Henri Poincare Probab.
Statist. 31 (1995), 143--154.

\bibitem[FV]{FV} P.L. Ferrari, B. Veto, ``Tracy-Widom asymptotics for q-TASEP'', Ann. Inst. H. Poincare Probab. Statist.51 (2015), 1465-1485.

%\bibitem[GSZ]{GSZ} P. Ghosal, A. Saenz, E. Zell, \textit{Limiting speed of a second class particle in ASEP}, preprint, {\tt arXiv:1903.09615}.

\bibitem[GO1]{GO1} A. Gnedin, G. Olshanski, ``q-exchangeability via quasi-invariance'', Ann. Probab.
Volume 38, Number 6 (2010), 2103-2135. {\tt arXiv:0907.3275}

\bibitem[GO2]{GO2} A. Gnedin, G. Olshanski, ``The two-sided infinite extension of the Mallows model for random permutations'', Advances in Applied Mathematics, Volume 48, Issue 5, May 2012, p. 615-639, {\tt arXiv:1103.1498}.
    
\bibitem[G]{G} S. Grosskinsky. Phase transitions in nonequilibrium stochastic particle systems with local conservation
laws. PhD thesis, PhD thesis, TU Munich, 2004.

\bibitem[H1]{H1} T. E. Harris, \textit{Additive Set-Valued Markov Processes and Graphical Methods}, Ann. Probab. 6 (1978), 355--378.

\bibitem[H2]{H2} T. E. Harris, \textit{Nearest-Neighbor Markov Interaction Processes on Multidimensional Lattices}, Adv. Math. 9 (1972), 66--89.

\bibitem[Hum]{Hum} J. E. Humphreys, ``Reflection groups and Coxeter groups'', Cambridge University Press, 1990

\bibitem[K]{K} S. Korotkikh, ``The Mallows measures on the hyperoctahedral group'', Journal of Mathematical Sciences (New York), 2017, 224:2, 269--277.

\bibitem[K2]{K2} J. Kuan, ``A multi-species ASEP(q,j) and q-TAZRP with stochastic duality'', preprint, \texttt{https://arxiv.org/pdf/1605.00691.pdf}.
    
\bibitem[K3]{K3} J. Kuan, ``Stochastic Fusion of Interacting Particle Systems and Duality
Functions'', preprint, {\tt arXiv:1908.02359}.
    
\bibitem[L]{L} T. M. Liggett. Ergodic theorems for the asymmetric simple exclusion process. Trans. Amer. Math. Soc.,
213:237--261, 1975

\bibitem[L1]{L1} T.~Lam, ``The shape of a random affine Weyl group element, and random core partitions'', Annals of Probability, 43 (2015), 1643-1662.

\bibitem[M]{M} C. L. Mallows, ``Non-null ranking models''. I. Biometrika 44 114–130, (1957).

\bibitem[M1]{M1} H. Malte, ``Reaction-diffusion processes and their connection with integrable quantum spin
chains''. Classical aand Quantum Nonlinear Integrable Systems: Theory and Application.
Edited by A. Kundu. Institute of Physics Series in Mathematical and Computational
Physics. (2003)

%\bibitem[MG]{MG} T. Mountford, H. Guiol, ``The motion of a second class particle for
%the TASEP starting from a decreasing shock profile'', Ann. Appl. Probab. 15
%1227–1259, (2005).

\bibitem[SW]{SW} T. Sasamoto, M. Wadati, ``Exact results for one-dimensional totally asymmetric diffusion models'', J. Phys. A: Math. Gen. 31 (1998), 6057--6071.


\end{thebibliography}
\end{document}